\documentclass{amsart}

\usepackage{amsmath,amsthm,amssymb,latexsym,amstext,amsfonts,mathabx}
\usepackage[english]{babel}
\usepackage{geometry}
\usepackage{graphics,comment}
\usepackage{MnSymbol} 
\usepackage{tikz}
\usetikzlibrary{%
  matrix,%
  calc,%
  arrows,%
  shapes,%
  positioning%
}
\usetikzlibrary{positioning}
\usepackage{hyperref}
 \usepackage{subfigure}
\usepackage{caption}
\usepackage{ifthen,xkeyval,xcolor,calc,graphicx}

\newtheorem{theorem}{Theorem}[section]

\newtheorem{proposition}[theorem]{Proposition}
\newtheorem{remark}[theorem]{Remark}

\newcommand{\Z}{\mathbb{Z}}
\newcommand{\R}{\mathbb{R}}

\newcommand{\X}{\mathbb{X}}
\newcommand{\Hg}{\mathrm{H}}
\newcommand{\rto}{\rightarrow}
\newcommand{\kk}{\mathrm{k}}

\newcommand{\set}[1]{\{\,#1\,\}}

\def \HH{{\mathcal P}}
\def \LL{{\mathcal L}}
\def \DD{{\mathcal D}}

\def \FF{{\mathcal F}}
\def \CC{{\mathcal C}}

\def \FF{{\mathcal F}}

\def \BB{{\mathcal B}}

\def \OO{{\mathcal O}}

\DeclareMathOperator{\dom}{dom}

\begin{document}

\title[Variable sets over an algebra of lifetimes]{\textbf{Variable sets over an algebra of lifetimes: \\ a contribution of lattice theory to the study of computational topology} \\ \medskip AAA88 Proceedings}

\author[JPC]{Jo\~ao Pita Costa *} 
\author[MVJ]{Mikael Vejdemo Johansson * **}
\author[P\v S]{Primo\v z \v Skraba *} 

\email{\underline{joao.pitacosta@ijs.si}, mvj@kth.se, primoz.skraba@ijs.si}

\address{* Laboratory of Artificial Intelligence, In\v stitut Jozef \v Stefan, Slovenia}
\address{** Computer Vision and Active Perception Laboratory, KTH Royal Institute of Technology, Sweden}

\date{\today}

\maketitle


%
\begin{abstract}
A topos theoretic generalisation of the category of sets allows for modelling spaces which vary according to time intervals.
Persistent homology, or more generally persistence, is a central tool in topological data analysis, which examines the structure of data through topology. 
The basic techniques have been extended in several different directions, encoding topological features by so called barcodes or equivalently persistence diagrams. 
The set of points of all such diagrams determines a complete Heyting algebra that can explain aspects of the relations between  persistent bars through the algebraic properties of its underlying lattice structure. 
In this paper, we investigate the topos of sheaves over such algebra, as well as discuss its construction and potential for a generalised simplicial homology over it.
In particular, we are interested in establishing a topos theoretic unifying theory for the various flavours of persistent homology that have emerged so far, providing a global perspective over the algebraic foundations of applied and computational topology. 



\end{abstract}


\section*{Introduction}
\label{Introduction}

Persistent homology is currently an area of research in computational and applied topology. 
The fundamental idea is to geometerize homology using multi-scale representations of spaces. 
One of the most prominent applications is \emph{topological data analysis}, studying the topology of point clouds as a route for approximating topological features of an underlying geometric object generating the samples. 
Since the identification of persistent homology as the homology of graded $k[t]$-modules in \cite{ZC05}, an algebraic approach has yielded immense benefits both in the innovation of new methods for topological data analysis and for algorithmic development. 
We identify the development of zigzag persistent homology \cite{carlsson_zigzag_2009}, as well as progress made on multidimensional persistence \cite{carlsson_theory_2009} as coming from fundamentally algebraic considerations.  
However, the approach of using more complicated rings, as in \cite{carlsson_theory_2009}, to model more general notions of persistence has raised a number of obstacles. 
In this paper, we adopt a different approach.  
We establish a foundation theory describing a general unifying framework for persistence with the construction of the appropriate topos of variable sets over an algebra of lifetime intervals $\HH$.
It is based on a theory of variable sets constructed over a lattice of time-like intervals of real numbers.
This will permit us to compute homology over a category with similar structure to the category of sets, that is \emph{parametrised} by the lifetime intervals leaving in the algebra $\HH$. 
The intuition is that  we can develop a set theory where all the sets have encoded a multiplicity of lifetimes determined by $\HH$.
In this setting, the topological features of a shape have lifetimes themselves. 
The goal of this approach is to place all flavours of persistence into a common framework, where some parameter of the framework determines the shape of the theory and the category in which analyses live.
For approachable introductions to the field and its applications, we recommend \cite{C09}, \cite{Ed10} and \cite{Gh08}, and for an accessible introduction to algebraic topology \cite{Hat00}.  

A sheaf of sets can be seen as a functor that is able to glue compatible local information providing us with a global perspective.  
A category of sheaves of sets is thus a collection of such functors and natural transformations between them.
On the other hand, the topos we are interested in is the category of sheaves over $\HH$.
The category of sets, a base of most of mathematical and mathematic flavoured constructions, can be generalised by such topos: 
it is the topos of sheaves of sets on the one point space, $\{ *\}.$
The usefulness of sheaves for probabilistic reasoning in quantum mechanics was recognised by Abramsky and Brandenburger in \cite{Abr11}, where they were able to generalise the no-go theorem by Kochen-Specker using a sheaf-based representation modelling contextually in quantum mechanics by sheaf-theoretic obstructions.
Based on this, D\"{o}ring and Isham establish a topos-theoretical foundation for quantum physics in \cite{Dor08} that is also described formally in the recent book of Flori on topos quantum theory in \cite{Flo13}. 
The inspiration for our approach is to some extent rooted in the presentation of \emph{time sheaves} in the exposition by Barr and Wells \cite{Bar95}, where sheaves of sets are described as sets with a particular shape given by a temporally varying structure. 
The shape corresponds to the shape of the underlying site, which also describes the shape of the available \emph{truth values} for the corresponding logic. 
Classical logic and set theory correspond to having two discrete truth values, while fuzzy logic corresponds to a continuum of truth values encoding reliability of a statement (e.g.: fuzzy sets are sheaves over $[0,1]$). 
Under this perspective, the persistent approach would encode truth as valid over some regions of a persistence parameter, but not other. 
In \cite{Bar95}, the authors give examples of time-like structures modelled by sheaf theory: over the total order $\R$ (sets have elements that arise and stay); and over intervals in $\R$ (sets have elements that are born and die).
The idea of applying sheaves to encode the shape of persistent homology is not itself new: it has been approached independently by McPherson and Patel \cite{Pat11}, and by Ghrist and Curry in \cite{Cur13} and \cite{Cur14}.
Though, this research provides us with an approach that can encode the various flavours of persistent homology through the internal logic of the persistence. 
We believe that topos-theoretic perspective can provide such a unifying theory.

The fundamental observation is that we have seen numerous cases lately where the \emph{shape} of a persistence theory matters; there has been the \emph{classical persistent homology} as defined in \cite{edelsbrunner_topological_2002}, \emph{multi-dimensional persistence} as defined in \cite{carlsson_theory_2009} and \emph{zig-zag persistence} as defined in \cite{carlsson_zigzag_2009}. 
In all of these cases, there is a sense of shape to the theory, embodied by a choice of algebra and module category that reflects the kinds of information we can extract from the method.
The similarities in definitions and in algorithms suggest to us that all three should be instances of a unifying theory; and indeed, one suggests itself directly from the definitions: in all three cases, we study homology for graded modules over graded rings (see \cite{ZC05, carlsson_theory_2009, carlsson_zigzag_2009, VJ12} for details). 
However, this similarity steps in on the algebraic plane; we are interested in a unifying theory that connects the underlying topological cases as well.
In \cite{VJ12}, M. Vejdemo-Johansson reviews in more detail the algebraic foundations of persistent homology, and presents the idea of a topos-based approach as a potential unifying language for these various approaches. 
In this paper we will show how to encode the lifetimes of topological features with an Heyting algebra $\HH$ determined in the space of all possible points in a persistence diagram. 
Then, we generalise the underlying set theory by the construction of a topos of sheaves over $\HH$ providing the basis for our unifying theory.
Such a topos can be seen as a category of sets with lifetime where things exist at some point and after a while might cease to exist.
Later, we discuss the computation of simplicial and semi-simplicial homology over such category of sets with lifetimes.  
In such a setting, the vertices of the simplexes have themselves lifetimes encoded in the underlying algebra $\HH$.
Hence, this seems to be a more appropriate universe to deal with problems in unified theory of persistence.


\section{Preliminaries}
\label{Preliminaries}

A \textbf{lattice} is a poset for which all pairs of elements have infimum and supremum.
Whenever every subset of a lattice $L$ has a supremum and a infimum, $L$ is 
named a \textbf{complete lattice}.
Every total order is a lattice. Though, not all of them are complete: $[0,1[$ with the usual order does not include the supremum of all its elements. 
A lattice $L$ can be seen as an algebraic structure $(L;\wedge ,\vee)$ with two associative, commutative and idempotent operations $\wedge$ and $\vee$ satisfying the absorption property, i.e., for all $x,y,z\in L$, $x\wedge (x\vee y)=x=x\vee (x\wedge y)$.
Moreover, $x\leq y$ if and only if $x\wedge y=x$ if and only if $x\vee y=y$, for all $x,y\in L$, providing the equivalence between the algebraic structure of a lattice $L$ and its ordered structure.
Given a lattice $L$, an \textbf{ideal} $I$ of $L$ is a nonempty subset of $L$ closed to $\vee$ such that, for all $x\in I$ and $y\in L$, $y\leq x$ implies $y\in I$. 
A \textbf{filter} is defined dually.
The ideal [filter] generated by a singleton $\set{x}$, with $x\in L$, is called \textbf{principal} ideal [filter] and denoted by $\downarrow x$ [$\uparrow x$]. 
An ideal [filter] $I$ of $L$ distinct from $L$ is called \textbf{proper}.
A proper ideal $I$ is \textbf{maximal} in $L$ if there is no other ideal in $L$ containing $I$. 
A proper ideal is \textbf{prime} if, for all $x,u,v\in L$, $x=u\wedge v$ implies $x=u$ or $x=v$.
A prime filter is defined dually. 
In fact, $I$ is a prime ideal if and only if $L/I$ is a prime filter.
An element $x$ of $L$ is \textbf{prime} if $\downarrow x$ is a principal prime ideal. 
An element $x$ of a lattice $L$ is \textbf{join-irreducible} if for all $p,q\in L$ such that $x=p\vee q$, we have $x=p$ or $x=q$.
In a distributive lattice $L$, a nonzero element $x\in L$ is join-irreducible if and only if $L\backslash \uparrow x$ is a prime ideal.
Dually, a nonzero element $x\in L$ is meet-irreducible if and only if $L\backslash \downarrow x$ is a prime filter (cf. \cite{Gr71}).

A lattice $L$ is \textbf{distributive} if, for all $x,y,z\in S$, it satisfies $x\wedge (y\vee z)=(x\wedge y)\vee (x\wedge z)$.
A \textbf{Boolean algebra} is a distributive lattice with a unary operation $\neg$ 
and nullary operations $\perp$ and $\top$ such that for all elements $x\in L$, $x\vee \perp = x$ and $x\wedge \top=x$; $x\vee \neg x = \top$ and  $x\wedge \neg x = \perp$.
A bounded lattice $L$ is a \textbf{Heyting algebra} if, for all $a,b\in  L$ there is a greatest element $x\in L$ such that $a\wedge x\leq b$. 
This element is the \textbf{relative pseudo-complement} of $a$ with respect to $b$ denoted by $a\Rightarrow b$. 
Please notice that we will distinguish the notation of this operation from the notation of logic implication, denoting the latter by a long right arrow $\Longrightarrow$.
A subalgebra of an Heyting algebra is thus closed to the usual lattice operations $\wedge $ and $\vee$, and to $\Rightarrow$. 
A homomorphism between Heyting algebras must preserve both lattice operations as well as the implication operation and both top and bottom elements. 
All the finite nonempty total orders (that are bounded and complete) constitute Heyting algebras, where $a\Rightarrow b$ equals $b$ whenever $a>b$, and $\top$ otherwise.
Every Boolean algebra is a Heyting algebra, with $a\Rightarrow b$ given by $\neg a \vee b$.
The lattice of open sets of a topological space $X$ forms a Heyting algebras under the operations of union $\cup$, empty set $\emptyset$, intersection $\cap$, whole space $X$, and the implication operation $U\Rightarrow V= \text{ interior of } (X - U)\cup V$.
A \textbf{complete Heyting algebra} is a Heyting algebra $\HH$ which constitutes a complete lattice.
It can also be characterised as any complete lattice satisfying the \textbf{infinite distributive law}, i.e., for all $x\in \HH$ and any family $\{ y_i\}_{i\in I}$ of elements of $\HH$, $x\wedge \bigvee_{i\in I}y_{i}=\bigvee_{i\in I}(x\wedge y_{i})$, with the implication operation given by $x \Rightarrow y = \bigvee \set{z\in \HH \mid z \wedge x \leq y}$, for all $x,y,z\in \HH$.
%

Given a complete Heyting algebra $\HH$ and a contravariant functor $\FF : \HH^{op} \rightarrow Set$ over the category of sets,  
a \textbf{compatible family} in $F$ is a family of elements in $\HH$ such that for each pair $x_i$ and $x_j$ of the restrictions of $s_i$ and $s_j$ agree on the overlaps, i.e., $s_{i \mid x_i\cap x_j} = s_{j \mid x_i\cap x_j}$.
Moreover, $F$ is a \textbf{sheaf} if, given $x = \vee_{i\in I} x_i$ in $\HH$ and $\{s_i \in F(x_i)\}_{i\in I}$ a compatible family in $F$, there exists a unique element $s \in F(x)$ such that for each index $i \in I$, $s_{\mid x_i}= s_i$.
Equivalently $\FF$ is a sheaf if it satisfies the following two conditions:
\begin{itemize}
\item[(i)] Given $x\in \HH$, if $(x_i)_{i\in I}$ is a family of elements in $\HH$ such that $\bigvee_{i\in I} x_i = x $, and if $s,t \in F(x)$ are such that $s_{\mid x_i} = t_{\mid x_i}$ for each $x_i$, then $s = t$ (we then call $\FF$ a \textbf{separated presheaf}); 
\item[(ii)] Given $x\in \HH$, if $(x_i)_{i\in I}$ is a family of elements in $\HH$ such that $\bigvee_{i\in I} x_i = x $, every compatible family $\{s_i \in F(x_i)\}_{i\in I}$ can be "glued" into a section $s \in F(x)$ such that $s_{\mid x_i}= s_i$ for each $i\in I$. 
\end{itemize}

The first condition is usually called \textbf{Locality} while the second is called \textbf{Gluing}.
By the first condition, $s$ is unique. 
Thus, both conditions together state that compatible sections can be uniquely glued together. 
For any objects $X$ and $Y$ in a category $\CC$ with all binary products with $Y$, an \textbf{exponential} object $Y^X$ is an object of $\CC$ equipped with an evaluation map $ev: Y^X \times X \to Y$ such that for any object $Z$ and map $e: Z \times X \to Y$ there exists a unique map $u: Z \to Y^X$ such that $Z \times X \stackrel{u \times id_X}\to Y^X \times X \stackrel{ev}\to Y$ equals the map $e$.
Whenever existing, it is unique up to unique isomorphism.
The \textbf{subobject classifier}, $\Omega$, is an object of $\CC$ and a monomorphism $true : * \to \Omega$ (where $*$ is the terminal object) such that for every monomorphism $U \to X$ in $\CC$ there is a unique morphism $\chi_U : X \to \Omega$ determining the correspondent pullback diagram.
%
%
A category $\CC$ is \textbf{Cartesian closed} if it has a terminal object and, for each pair of objects $X$ and $Y$ in $\CC$, the product $X\times Y$ and the exponential object $Y^X$ exist.
Any Heyting algebra $\HH$ seen as a thin small (poset) category is Cartesian closed: for all $x,y\in \HH$, the product of $x$ and $y$ is $x\wedge y$ and the exponential $x^y$ is $x\Rightarrow y$.
%
A \textbf{topos} is a Cartesian closed category with all finite limits and a subobject classifier.  
The category $Set$ is a topos with subobject classifier $\Omega = \{0,1\}$. 
Roughly speaking, any topos behaves as a category of sheaves of sets on a topological space.
If $\X$ is a topological space, the category of sheaves over $\X$, $Sh(\X)$, is a topos with $\Omega(U ) = \{V \mid V \subset U \}$.
In general, the category of sheaves over a complete Heyting algebra is a topos (cf. \cite{Joh02}).
A good review on Heyting algebras, sheaf theory and, in particular, on topos theory can be found in \cite{Aw06}, \cite{Mac92} and \cite{Joh02}.


\section{Motivation on Persistent Homology}
\label{Motivation}

As motivation for this research, we describe some aspects of computational topology with a focus on the computation of persistent homology, as well as review some of the variants.
Persistent homology permits us to recover topological information by applying geometric tools followed by methods from algebraic topology to obtain a topological descriptor. 
In topological data analysis we often view data as a finite metric space, build complexes of points (most often \v Cech or Vietoris-Rips), and analyse the topology of these objects to infer the topology of that data. 
Recall that, due to the nerve theorem, the \v Cech complex associated with any covering of the space with balls is homotopy equivalent to the original space.
The construction of these complexes requires a choice of parameter (such as the radius of the balls for the \v Cech complex).  
Persistent homology lets the parameter value vary while tracking the births and deaths of topological features.
The output, in the standard case from \cite{edelsbrunner_topological_2002} is a set of intervals on the real line that can also be encoded as points in a persistence diagram. This measures the significance of a topological feature.
Usually additional restrictions are imposed to ensure that the homology changes at only finitely many values.
Many of these restrictions can be relaxed, defining persistence diagrams in a wide variety of situations as in \cite{Fre12}.

A \textbf{$p$-dimensional homology class} is an equivalence of \textbf{$p$-cycles}, i.e., a collection of mutually homologous points (in dimension $p=0$), closed curves (in dimension $p=1$) or closed surfaces (in dimension $p=2$) in $\X$. 
The \textbf{$p$-th homology group} of the space $\X$ is the vector space $H_p(\X)$ of all $p$-dimensional homology classes with rank $\beta _p(\X)$, the \textbf{$p$-th Betti number} of $\X$.  
If $\X_0$ is a subspace of $\X$, the \textbf{$p$-th relative homology group} of the pair of spaces $(\X,\X_0)$ is the vector space $H_p(\X,\X_0)$ of all $p$-dimensional relative homology classes with rank $\beta _p(\X)$, the \textbf{$p$-th relative Betti number} of $(\X,\X_0)$.  
The \textbf{essential classes} are the ones that represent the homology of $\X$ while all others are called \textbf{inessential classes}.
As homology classes do not come with a notion of size, persistent homology takes a compact topological space $\X$ along with a real-valued (height) function $f$ and returns the size, as measured by $f$, of each homology class in $\X$.

Let $\X$ be a space and $f:\X \rto \R$ a real function.  
We denote $\oplus_i H_i(\X)$ by $H_*(\X)$.
The object of study of persistent homology is a \textbf{filtration} of $\X$, i.e., a monotonically non-decreasing sequence
\begin{equation*}
\emptyset = \X_0 \subseteq \X_1 \subseteq \X_2 \subseteq \ldots\subseteq\X_{n-1} \subseteq \X_n = \X
\end{equation*}
determined by the (height) function $f$ as follows: $X_i=f^{-1}(]- \infty, t])$. 
As $t$ runs in $]-\infty,+\infty[$, the sublevel sets $\X_i$ include into one another and get bigger, eventually forming the space $\X$ itself, while homology classes appear and disappear. Persistent homology is the homology of a filtration, tracking and quantifying the described evolution.
To simplify the exposition, we assume that this is a discrete finite filtration of tame spaces. Taking the homology of each of the associated chain complexes, we obtain 
\begin{equation*}
 \Hg_*(\X_0) \rto  \Hg_*(\X_1) \rto  \Hg_*(\X_2) \rto \ldots\rto \Hg_*(\X_{n-1}) \rto  \Hg_*(\X_n).
\end{equation*}
We take homology over a field $\kk$ -- therefore the resulting homology groups are vector spaces and the induced maps are linear maps. 
The standard persistent homology module $H_*(\X)$ describes how the absolute homology groups $H_*(X_i)$ relate to each other as $i$ varies.
Due to a version of Alexander duality \cite{Dual11}, a similar description is possible for the absolute cohomology groups $H^*(X_i)$, the relative homology groups $H_*(X_n,X_i)$, and the relative cohomology groups $H^*(X_n,X_i)$ as represented below:
\begin{center}
$\begin{array}{lcr}
H_*(\X): &  H_*(X_1)\rightarrow \hdots \rightarrow H_*(X_{n-1})\rightarrow H_*(X_n) \\
H^*(\X): &  H^*(X_1)\rightarrow \hdots \rightarrow H^*(X_{n-1})\rightarrow H^*(X_n) \\
H_*(X_{\infty},\X): & H_*(X_n)\rightarrow H_*(X_n,X_1)\rightarrow \hdots \rightarrow H_*(X_n,X_{n-1})  \\
H^*(X_{\infty},\X): & H^*(X_n)\rightarrow H^*(X_n,X_1)\rightarrow \hdots \rightarrow H^*(X_n,X_{n-1})    
\end{array}$
\end{center}
To each homology class is assigned a \textbf{lifetime} encoded by an interval with endpoints in the real line. To each dimension, encoded as a non-negative integer k, corresponds a \textbf{persistence barcode}, by which we mean the finite collection of lifetimes of the homology classes appearing in its filtration. The integer $k$ specifies a dimension of a feature (zero-dimensional for a cluster, one-dimensional for a loop, etc.), and any bar in the barcode represents a feature which is \textbf{born} at the value of a parameter given by the left hand endpoint of the interval, and which \textbf{dies} at the value given by the right hand endpoint. 
A \textbf{persistence diagram} is a multi-set of points of $\R\times \R$ representing a persistence barcode by corresponding a bar in the barcode with birth time $x$ and death time $y$ by the point $(x,y)$. 
 
The persistence diagram for absolute cohomology (as for relative homology and cohomology) is also a multi-set of integer ordered pairs.
Moreover, homology and cohomology have identical barcodes, while persistent homology and relative homology barcodes carry out the same information, with a dimension shift for the finite intervals.
Thus, provided we take the dimension shifts into account, all four barcodes carry exactly the same information (cf. \cite{Dual11}).

Given a real-valued function $f:\X \rightarrow \R$, we call \textbf{extended persistence} to the collection of pairs arising from a sequence of absolute and relative homology groups.
The correspondent pairs in the extended persistence diagram keep track on the changes in the homology of the input function.
As in \cite{Be12}, we consider persistent homology as a two-stage filtering process:
\begin{itemize}
\item[]\textbf{standard persistence}: In the first stage, filter $\X$ via the sub level sets $\X_r= f^{-1}(-\infty,r]$ of $f$, where $r\in \R$; 
\item[]\textbf{extended persistence}: In the second stage, we consider pairs of spaces $(\X,\X^r)$, where $\X^r=f^{-1}[r,\infty)$ is a super level set and $r\in \R$.
\end{itemize}

Every class which is born at some point of the two-stage process will eventually die, being associated with a pair of critical values.
These pairs fall into three types:
\begin{itemize}
\item[(i)] \textbf{ordinary pairs} : have birth $x$ and death $y$ during the first stage, being represented in the persistence diagram by a point $(x,y)$ with $x<y$;
\item[(ii)] \textbf{relative pairs} : have birth $x$ and death $y$ during the second stage, being represented in the persistence diagram by a point $(x,y)$ with $y<x$;
\item[(iii)] \textbf{extended pairs} : have birth $x$ in one stage and death $y$ in the other (their representation in the persistence diagram will coincide with one of the two cases above as seen in Fig. \ref{extended1}.
\end{itemize}
In Figure \ref{extended1} it is represented a version of the classical case of a torus with a height function from \cite{Be12}, together with the correspondent barcode comprehending the ordinary classes given by the bars in a positive (upwards) direction; the relative classes given by the bars in the negative (downwards) direction; and the extended classes in which the bars go to infinity and then come back. 
Also in the same figure its represented (on the right) the traditional persistence diagram that tracks the  topological information given by the barcodes.
Notice that the ordinary classes are represented by points $(x,y)$ where $x<y$, while relative classes are represented by points $(x,y)$ where $y<x$. 

\begin{figure}
\includegraphics[scale=0.5, page=1,width=0.6\textwidth]{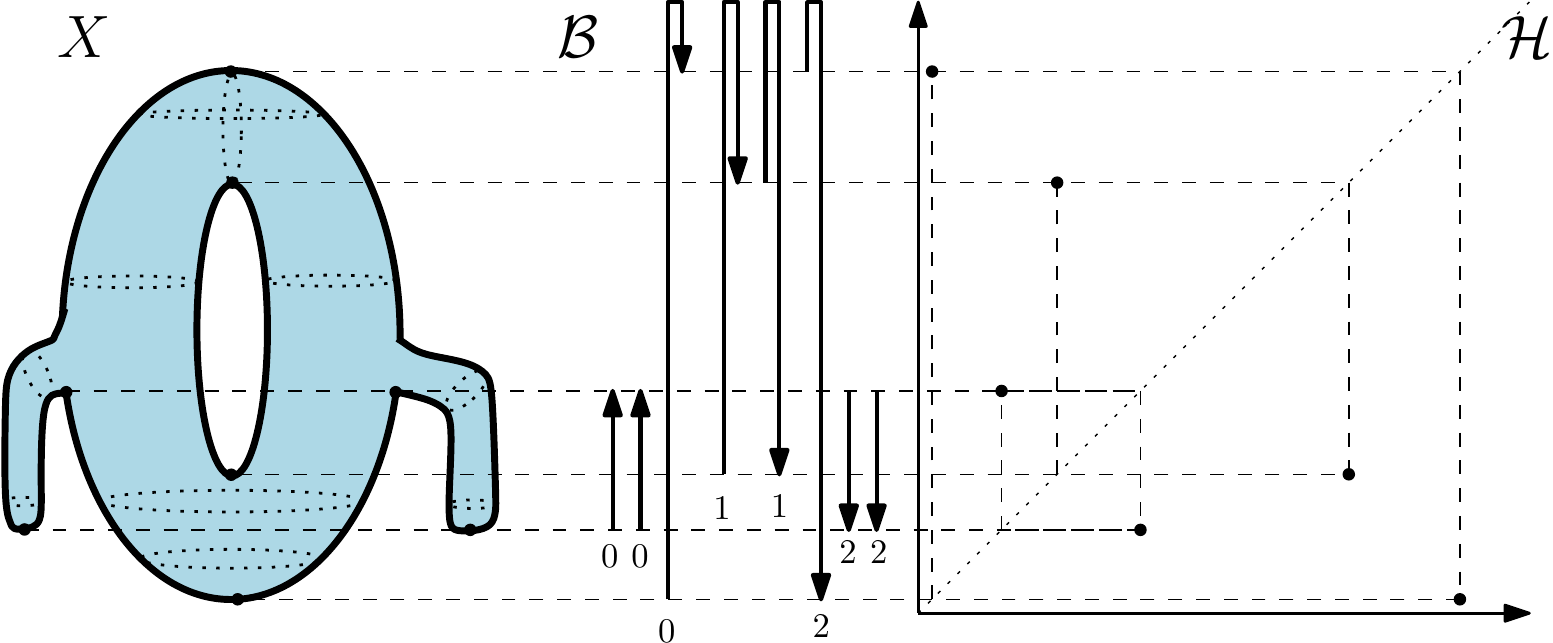} 
\caption{
The height function of a topological space $X$, its correspondent extended barcode $\mathcal{B}$ and the persistence diagram $\HH$.
The numbers below the bars correspond to the dimensions they refer to. 
Notice that in this example there are two elements of the persistence diagram with multiplicity 2, all others have multiplicity 1.
}\label{extended1}
\end{figure}

\begin{remark}\label{assumption}
It is possible to distinguish both the multiplicity of an element in a persistence diagram, or the indication that such element corresponds to an extended bar or not. 
Though, in the following sections we shall consider the space of all ordinary and relative pairs in a persistence diagram, ignoring their multiplicity and identifying these pairs with the extended pairs that have the same coordinates.
We shall distinguish between relative and ordinary pairs, corresponding to bars with different orientation.
Other possible directions of research point out to have the algebra of lifetimes, described in the next section, being build over aspects of total and pointwise existence of persistence. 
This provides new arguments for the choice of this model and is a research topic by itself to be developed in further steps. 
\end{remark}

In the variants described above, the methods are tied to the total order of the reals $\R$.
Generalisations of persistence such as zigzag persistence \cite{Cle15} or multidimensional persistence \cite{Les15} do consider other oreder underlying structures. However, these theories are much less developed. 
The purpose of this investigation is to explore the underlying lattice structure and introduce a unifying framework where these (and other) generalisations of persistence can find a common language.


\section{The algebra of lifetimes}
\label{The algebra of lifetimes}
Our interest is to model the algebra of barcodes used in the methodology of persistent homology, considering non disjoint intervals, i.e., \emph{time indexed sets}. 
In this, given a time indexed set $t$, $\bigwedge_{i\in I} t_i$ is the birth time while $\bigvee_{i\in I} t_i$ is the death time. 
With this in mind, the set of all intervals of real numbers, ordered by set inclusion, corresponds to the algebra of open sets of the topological space $\R$ and constitutes a complete Heyting algebra for the operations of set theoretic intersection and union. 
However it also includes disjoint intervals which are not of interest in our model, in the sense that a lifetime should correspond to a closed interval of the real line, having a birth time and a death time.
Fix $\varepsilon \in \R^+$ and consider the total order in the positive real numbers no bigger than $\varepsilon$ including zero, i.e., the complete lattice $([0,\varepsilon]; \wedge, \vee)$ where $x\wedge y=inf\{x,y\}$ and $x\vee y=sup\{x,y\}$, for any $x,y\in \R$.
If we substitute set theoretic union by its cover (i.e., $[a_1,a_2]\vee [b_1,b_2]=[a_{1}\wedge b_{1},a_{2}\vee b_{2}]$), we get a complete lattice. 
It unfortunately does not constitute a Heyting algebra as distributively fails:
to see this consider the intervals $a=[0,0.2]$, $b=[0, 0.5]$ and $c=[0.6,1]$ and observe that $a\subseteq b$ and that
\[
(a\vee c)\wedge b = [0,1]\wedge b = b \neq a = a\vee \emptyset = (a\wedge b)\vee (c\wedge b).
\]
Consider now the representation of barcodes in a persistence diagram. 
Let $\HH$ be the quarter plane of all the points in all possible persistence diagrams bounded by $(0,0)$ and $(\varepsilon_1,\varepsilon_2)$. 
Let $a=(a_1,a_2)$ be a point in a persistence diagram and denote by $\BB(a)$ the correspondent interval $[a_1,a_2]$, where $a_1$ is the birth time and $a_2$ is the death time. 
Recall from Remark \ref{assumption} that we are also considering points of the persistence diagram with coordinates such that $a_2<a_1$.
%
Consider the set $\HH$ of points of all possible persistence diagrams.
Then, $\HH$ with the operations
\[
a\wedge b = (a_1\vee b_1,a_2\wedge b_2) \text{   and   }  a\vee b = (a_1\wedge b_1,a_2\vee b_2)
\]
constitutes a lattice ordered by
\[
a\leq b \text{  if and only if  } b_{1}\leq a_{1} \text{  and  } a_{2}\leq b_{2},
\]
for all $a=(a_1,a_2),b=(b_1,b_2)\in \HH$.
%
%

\begin{figure}
\subfigure[lattice operations $\wedge$ and $\vee$ for related elements]{
\includegraphics[page=1,width=0.45\textwidth]{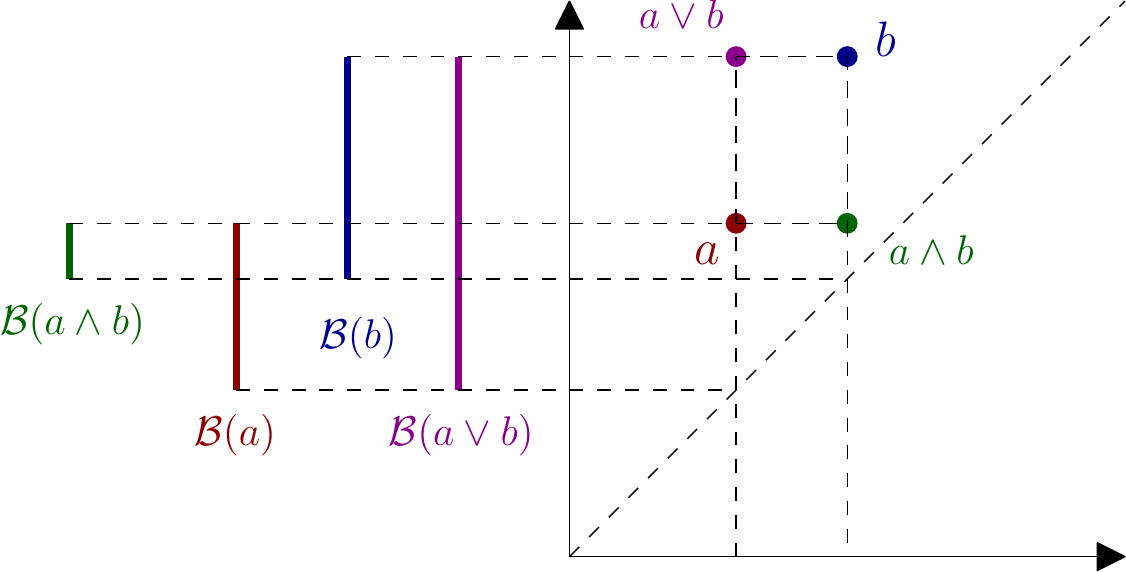} 
}
\subfigure[lattice operations $\wedge$ and $\vee$ for unrelated elements]{
\includegraphics[page=1,width=0.45\textwidth]{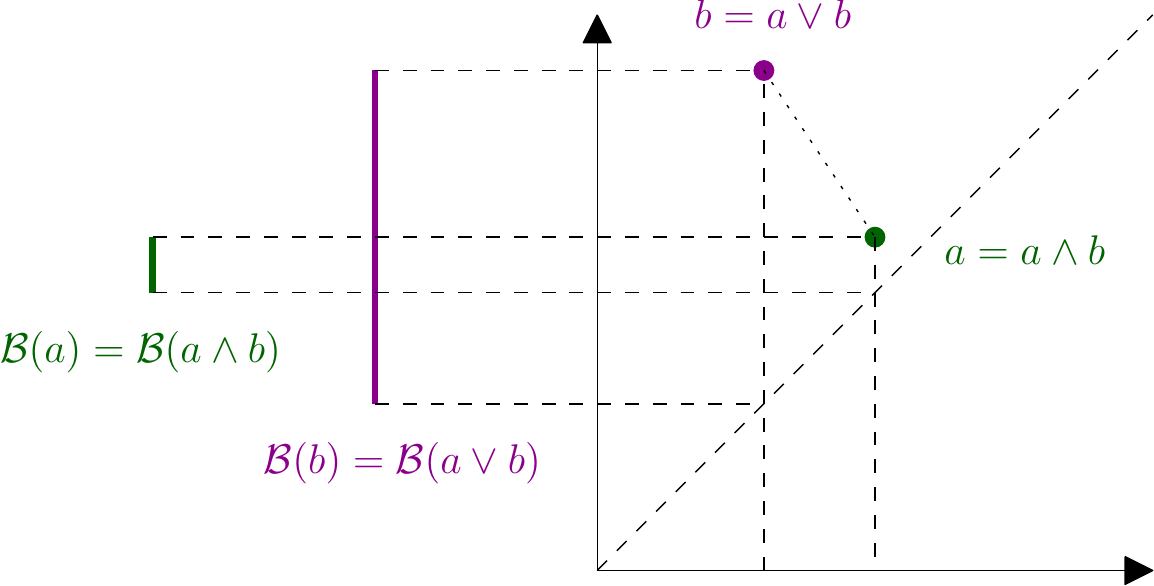} 
}
\caption{The representation of the lattice operations $\wedge$ and $\vee$ in the algebra of lifetimes $\HH$.}
\label{figjoinmeet}
\end{figure}

These operations give to $\HH$ the algebraic structure of a lattice, named \textbf{the algebra of lifetimes}.
In fact considering the non negative real numbers in $[0,\varepsilon_i]$ for $i\in\{1,2\}$ with the usual partial order as the poset category where a morphism $a\rightarrow b$ is just $a\leq b$, $\HH$ is given by the product $[0,\varepsilon_1]^{op} \times [0,\varepsilon_2]$.
We will not distinguish the notation of theses lattice operations neither the correspondent order from the notations for lattice operations and order in $[0,\varepsilon]$ whenever the context is clear. 
Figure \ref{figjoinmeet} shows a representation of those operations from $\HH$ in the extended barcode and in the persistence diagram.

Each element $x$ of $\HH$ is an ordered pair $(x_1,x_2)$ that can be seen as a \textbf{generalised time interval} with:
\begin{itemize} 
\item[(i)]  a \textbf{length}, given by the absolute value of $x_2-x_1$ that corresponds to the persistence of the measured topological feature;
\item[(ii)]  an \textbf{orientation}: \textbf{positive} if $x_1\leq x_2$ and \textbf{negative} if $x_2\leq x_1$.
\end{itemize}
The degenerate case where $x_1=x_2$ corresponds to a lifetime of length zero. 

The partial order determined by the lattice operations shows us how the operations are indeed natural: the derived order structure corresponds to the inclusion of the correspondent bars for the upper triangle where death times are smaller then birth times: for all $x,y\in \HH$ such that $x<y$,
\[
x\leq y \text{  if and only if   } \BB(x)\subseteq \BB(y).
\]
\emph{This is not a total order}: unrelated bars $\BB(a_1,a_2)$ and $\BB(b_1,b_2)$ are of such sort that $a_1\leq b_1$ and $a_2\leq b_2$.
Moreover, \emph{the smallest element} $\perp$ is in the right lower corner of the diagram, correspondent to the point $(\varepsilon_1,0)$, while \emph{the biggest element} $\top$ is on the left upper corner, correspondent to the point $(0,\varepsilon_2)$. 
Furthermore, \emph{if one of the coordinates is equal, then the correspondent bars are always related}. 
Hence, bars with equal death time or birth time must be included in one another. 

\begin{proposition}\label{infdist}
The algebra of $\HH$, together with the extended operations
\[
\bigwedge_{i\in I} a_i = (\bigvee _{i\in I} a_{1i},\bigwedge a_{2i}) \text{  and  } \bigvee_{i\in I} a_i = (\bigwedge_{i\in I} a_{1i},\bigvee a_{2i})
\]
determines a complete lattice satisfying the infinite distributive law given by the following identity:
\[
x\wedge {\bigvee_{i\in I}y_{i}}={\bigvee_{i\in I}(x\wedge y_{i})},
\]
for all $x\in \HH$ and any family $\{y_i\}_{i\in I}$ of elements of $\HH$.
Hence, $\HH$ is a complete Heyting algebra.
\end{proposition}

\begin{proof}
By construction, the lattice operations can naturally be generalised from pairs of elements to any set of elements.
The infinite distributivity law follows directly from the definition of the lattice operations together with the fact that $([\varepsilon_1,\varepsilon_2];\wedge,\vee)$ constitutes a completely distributive lattice. 
To see this just observe that: 
\[
x\wedge {\bigvee_{i\in I}y_{i}} = (x_1\vee (\bigwedge_{i\in I} y_{1i}),x_2\wedge (\bigvee_{i\in I} y_{2i}) ) = (\bigwedge_{i\in I} (x_1\vee y_i), \bigvee_{i\in I}(x_2\wedge y_{2i}) ) = {\bigvee_{i\in I}(x\wedge y_{i})}. 
\]
The fact that $\HH$ is a Heyting algebra follows from $\HH$ being a complete and distributive lattice satisfying the infinite distributivity law (cf. \cite{Joh86}).
\end{proof}

Let us now examine the implication operation. 
The following result describes this operation for any of the possible cases, represented in Figure \ref{figarrows}, both in the context of an extended barcode or a persistence diagram.

\begin{proposition}\label{galois1}
Let $a,b\in \HH$. Then, if $a$ and $b$ are (order) related,
\[
a\Rightarrow b=\begin{cases} %
\top=(0,\varepsilon_2) & \text{,  if  }  b_1\leq a_1 \text{  and  } a_2\leq b_2  \\ %
b=(b_1, b_2) & \text{,  if  }  a_1\leq b_1\text{  and  } b_2\leq a_2 %
 \end{cases}.
\]
Otherwise,
\[
a\Rightarrow b=\begin{cases} %
(b_1, \varepsilon_2) & \text{,  if  }  a_1\leq b_1 \text{  and  } a_2\leq b_2  \\ %
(0, b_2) & \text{,  if  }  b_1\leq a_1\text{  and  } b_2\leq a_2 %
 \end{cases}.
\]
\end{proposition}

\begin{proof}
The fact that the persistence diagram $\HH$ is a complete distributive lattice implies that the implication operation is defined for every $a,b\in \HH$ as follows:
\[
a\Rightarrow b = \bigvee \{x \in \HH \mid  x\wedge a\leq b\}=\bigvee \{x \in \HH \mid  x\wedge a\wedge b=x\wedge a\},
\]
that is,
\[
a\Rightarrow b = \bigvee \{(x_1,x_2) \in \HH \mid  a_1\vee b_1\vee x_1= a_1\vee x_1 \text{  and  }  a_2\wedge b_2\wedge x_2=a_2\wedge x_2\},
\]
that is,
\[
a\Rightarrow b = \{(\bigvee_{i\in I} x_{1i}, \bigwedge_{i\in I} x_{2i}) \mid    b_1\leq a_1\vee x_1 \text{  and  }  a_2\wedge x_2\leq b_2 \}.
\]

\emph{Case 1:} When $a\leq b$ we get that $b_1\leq a_1$ and $a_2\leq b_2$ so that $x_1$ and $x_2$ can take any value. 
The biggest bar in these conditions has the smallest $x_1$ and the biggest $x_2$, thus coinciding with the maximum element $\top$.

\emph{Case 2:} On the other hand, if $b\leq a$ then $a_1\leq b_1$ and $b_2\leq a_2$ so that $a\wedge x\leq b$ can only hold when $b_1\leq x_1$ and $x_2\leq b_2$. 
The biggest interval in this conditions corresponds to $x_1=b_1$ and $x_2=b_2$. 

\emph{Case 3:} When $a_1\leq b_1$ and $a_2\leq b_2$, then $a_1\in \set{x_{1i}\mid a_1\wedge x_{i1}\leq b_1}_{i\in I}$ and $a_2\vee x_{i2}=x_{i2}$ so that 
\[
a\Rightarrow b =(\bigvee_{i\in I} x_{1i}, \bigwedge_{i\in I} x_{2i}) \text{,  such that  }  b_2 \leq x_{i2}=(0,b_2).
\]

\emph{Case 4:} Similarly, whenever $b_1\leq a_1$ and $b_2\leq a_2$, then $a_1\wedge x_{i1}=x_{i1}$ and $a_2\in \set{x_{2i}\mid b_2\leq a_2\vee x_{i2}}_{i\in I}$.
Therefore, 
\[
a\Rightarrow b =(\bigvee_{i\in I} x_{1i}, \bigwedge_{i\in I} x_{2i}) \text{,  such that  }  x_{i1}\leq b_1=(b_1,0).
\]
\end{proof}

\begin{remark}
Due to the completeness of the underlying lattice structure, the operations $\wedge $ and $\rightarrow$ are adjoints in two suitable monotone Galois connections. 
Particularly, the fact that $\HH$ is a Heyting algebra implies that, given $a\in \HH$, the mapping $\varphi_{a}:\HH \rightarrow \HH$ defined by $\varphi_{a} (x)=a\wedge x$ is the lower adjoint of a Galois connection with respective adjoint $\psi_{a}: \HH \rightarrow \HH$ defined by $\psi_{a}(x) = a\Rightarrow x$. 
This Galois connection defines a pair of dual topologies so that such topologies can be defined by means of binary relations (cf. \cite{Bor94}).
Given $X=(x_1,x_2)\in \HH$, $x\wedge a\leq b$ if and only if $x\leq (a\Rightarrow b)$.
But $x\leq (a\Rightarrow b)$ means that $b_1\leq a_1\vee x_1$ and $a_2\wedge x_2\leq b_2$.
Thus
\[
\begin{cases} %
0\leq x_1,x_2\leq \varepsilon_2 & \text{,  if  }  b_1\leq a_1 \text{  and  } a_2\leq b_2  \\ %
b_1\leq x_1, x_2\leq b_2 & \text{,  if  }  a_1\leq b_1\text{  and  } b_2\leq a_2 \\ %
b_1\leq x_1, x_2\leq \varepsilon_2 & \text{,  if  }  a_1\leq b_1 \text{  and  } a_2\leq b_2  \\ %
0\leq x_1, x_2\leq b_2 & \text{,  if  }  b_1\leq a_1\text{  and  }b_2\leq a_2 %
 \end{cases}
\]
Then, for all $x\in \HH$, 
\[
(x\wedge a)\leq b \text{    iff    } 
\begin{cases} %
a\leq b \\ %
x\leq b & \text{,  if  }  b\leq a \\ %
b_1\leq x_1 & \text{,  if  }  a_1\leq b_1 \text{  and  } a_2\leq b_2  \\ %
x_2\leq b_2 & \text{,  if  }  b_1\leq a_1\text{  and  }b_2\leq a_2 %
 \end{cases}
\]
These are thus the conditions for candidates in the algebra $\HH$ to be the element $a\Rightarrow b$. 
\end{remark}

\begin{figure}
\subfigure[arrow between related elements]{
\includegraphics[scale=0.5, page=1,width=0.45\textwidth]{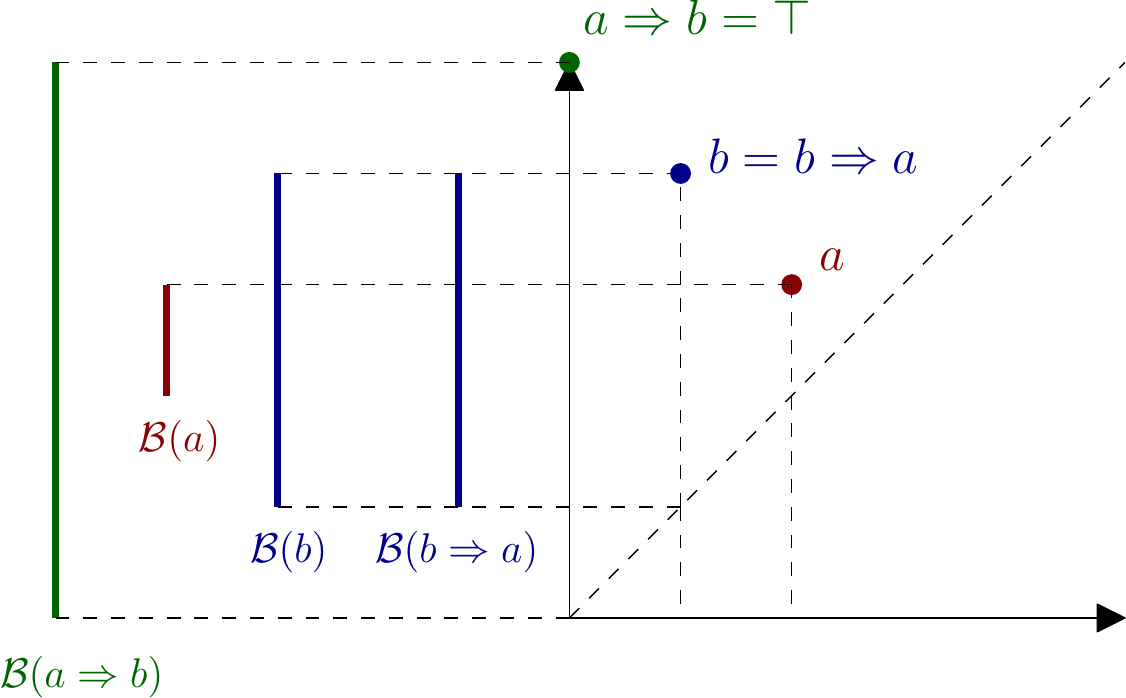}

}
\subfigure[arrow between unrelated elements]{
\includegraphics[scale=0.5, page=1,width=0.45\textwidth]{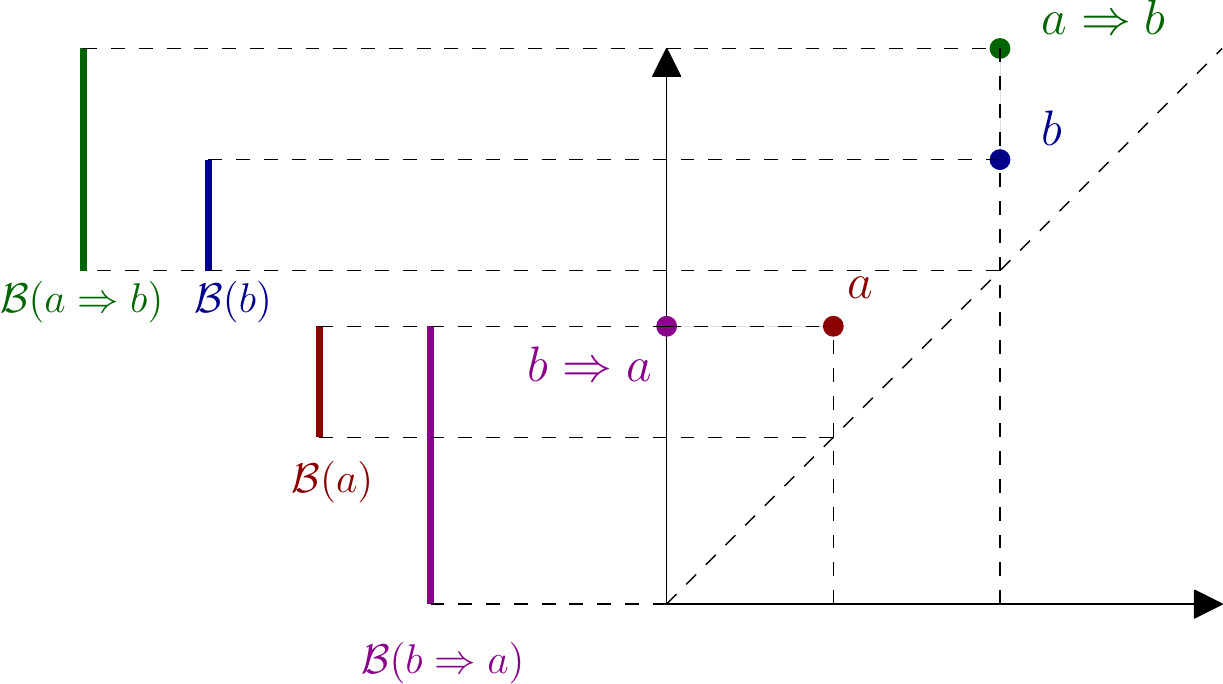}
}
\caption{Arrow operation between related and unrelated elements represented both in the algebra of lifetimes $\HH$.}
\label{figarrows}
\end{figure} 
 
\begin{remark}\label{galoish}
Observe that the algebra of lifetimes $\HH$ has a largest element $\top$ given by $x\Rightarrow x$, for any $x\in \HH$.
Moreover, $\HH$ is not a Boolean algebra: indeed, take $a=(a_1,a_2)$ such that $a_1,a_2\neq 0$, $a_1\neq \varepsilon_1$ and $a_2\neq \varepsilon_2$ (recall that $\varepsilon_1$ and $\varepsilon_2$ are the biggest values on the $X$ coordinates and $Y$ coordinates). 
Then, for all $b=(b_1,b_2)$, $a\wedge b=\perp$ implies $\varepsilon_1=a_1\vee b_1$ and $0=a_2\wedge b_2$, that is, $\varepsilon_1=b_1$ and $0=b_2$.
Similarly, $a\vee b=\top$ implies $0=a_1\wedge b_1$ and $\varepsilon_2=a_2\vee b_2$, that is, $0=b_1$ and $\varepsilon_2=b_2$.
Hence, $a$ has no complement in $\HH$.
In fact, the only complemented elements in $\HH$ are $(\varepsilon_1,0)$, $(0,0)$, $(\varepsilon_1,\varepsilon_2)$ and $(0,\varepsilon_2)$.
\end{remark}

%
%

The following paragraphs describe the order structure of $\HH$ through the study of its ideals and filters. 
Towards the end of this section, we shall also discuss aspects of the dual space for $\HH$ in the light of Stone duality.
An ideal of the algebra of lifetimes $\HH$ is any downset $I$ that constitutes a subalgebra, i.e., any subset $I$ closed to the operation $\vee$ such that, for all $x\in \HH$ and $y\in I$, $x\leq y$ implies $x\in \HH$. 
Notice that an ideal of $\HH$ closed to arbitrary joins must be determined by one unique element, i.e., it must be a principal ideal. 
Given an arbitrary element $a=(a_1,a_2)$, the ideal it generates is the following set: 
\[
\downarrow a = \set{x \in \HH \mid x \leq a} = \set{x \in \HH \mid a\vee x =a} = \set{(x_1,x_2)\in \HH \mid x_1\wedge a_1=a_1 \text{  and  } x_2\vee a_2=a_2} 
\]
\[
= \set{(x_1,x_2)\in \HH \mid a_1\leq x_1 \text{  and  }  x_2\leq a_2}  =  [a_1,\varepsilon_1]\times[0, a_2].
\]
Dually, the principal filter generated by $a=(a_1,a_2)$ is $\uparrow a=[0, a_1]\times [a_2,\varepsilon_2].$

%


\begin{proposition}\label{genideal}
The ideal generated by two elements $a=(a_1,a_2)$ and  $b=(b_1,b_2)$ of $\HH$, $\downarrow \set{a,b}$, is the ideal generated by $a\wedge b$, i.e., $[a_1\vee b_1,\varepsilon_1]\times [0, a_2\wedge b_2].$
Dually, the filter generated by two elements $a=(a_1,a_2)$ and  $b=(b_1,b_2)$ of $\HH$, $\uparrow \set{a,b}$, is the filter generated by $a\vee b$, i.e., $ [0, a_1\wedge b_1]\times [a_2\vee b_2,\varepsilon_2].$
In general, for any family of elements $\set{a_i}_{i\in I}$ of $\HH$, we get that:
\[
\downarrow  \set{a_i}_{i\in I}=\downarrow (\bigwedge_{i\in I}  a_i) =[\bigvee_{i\in I} \set{a_i}_{i\in I},\varepsilon_1]\times [0, \bigwedge_{i\in I} \set{a_i}_{i\in I}]
\]
\[
 \text{  and  } \uparrow \set{a_i}_{i\in I}=\uparrow (\bigvee_{i\in I}  a_i)=[0, \bigwedge_{i\in I} \set{a_i}_{i\in I}]\times [\bigvee_{i\in I} \set{a_i}_{i\in I},\varepsilon_2]. 
\]
\end{proposition}

\begin{proof}
Take two arbitrary elements $a=(a_1,a_2)$ and  $b=(b_1,b_2)$ of $\HH$. 
Let $(x_1,x_2),(x'_1,x'_2)\in [a_1\wedge b_1,\varepsilon_1]\times [0, a_2\vee b_2]$ and $(y_1,y_2)\in \HH$.
Then $a_1\wedge b_1 \leq x_1\wedge x'_1$ and $x_2\wedge x'_2\leq a_2\vee b_2$, so that $(x_1,x_2)\vee (x'_1,x'_2)\in [a_1\wedge b_1,\varepsilon_1]\times [0, a_2\vee b_2]$.
On the other hand, if $(y_1,y_2)\leq (x_1,x_2)$ then $a_1\wedge b_1\leq x_1\leq y_1$ and $y_2\leq x_2\leq a_2\vee b_2 $, so that $(y_1,y_2)\in [a_1\wedge b_1,\varepsilon_1]\times [0, a_2\vee b_2]$
The dual result has a similar proof and the general case can be proven by induction on the number of generating elements, taking in account the completeness of the lattice.
\end{proof}


\begin{proposition}\label{joinirred}
The join-irreducible elements of $\HH$ are all the elements $(x_1,0)$ or $(\varepsilon_1,x_2)$, with $0\leq x_1<\varepsilon_1$ and $0\leq x_2<\varepsilon_2$.
Dually, the meet-irreducible elements are all the elements $(0,x_2)$ and $(x_1,\varepsilon_2)$, with $0< x_1\leq \varepsilon_1$ and $0< x_2\leq \varepsilon_2$. 
\end{proposition}

\begin{proof}
Given an arbitrary element $a=(a_1,a_2)$, the elements $b=(a_1-1,a_2)$ and $c=(a_1,a_2+1)$ are such that $a=b\vee c$.
Thus, if $b$ and $c$ are distinct from $a$, then $a$ is not a join-irreducible element. 
Hence, the candidates for join-irreducible elements are the bars $a=(a_1,a_2)$ such that $a_1$ is the biggest first coordinate, i.e., $a_1=\varepsilon_1$; or $a_2$ is the  least second coordinate, i.e., $a_2=0$. 
Let us see that this indeed is the case. Take $x=(x_1,0)$ such that $x=a\vee b$ for any $a=(a_1,a_2)$ and $b=(b_1,b_2)$ in $\HH$. 
Then $0=a_2\vee b_2$, i.e., $a_2=b_2=0$ and $x=a_1\wedge b_1$ and thus $x=a_1$ or $x=b_1$ so that $x=a$ or $x=b$.
Similarly, $y=(\varepsilon_1,y_2)=(a_1,a_2)\vee (b_1,b_2)$ implies $\varepsilon_1=a_1\wedge b_1$, i.e., $a_1=\varepsilon_1=b_1$. 
On the other hand, $y_2=a_2$ or $y_2=b_2$ so that $y=a$ or $y=b$.
The proof regarding meet-irreducible elements is analogous.
\end{proof}

\begin{figure}
\includegraphics[page=1,width=0.6\textwidth]{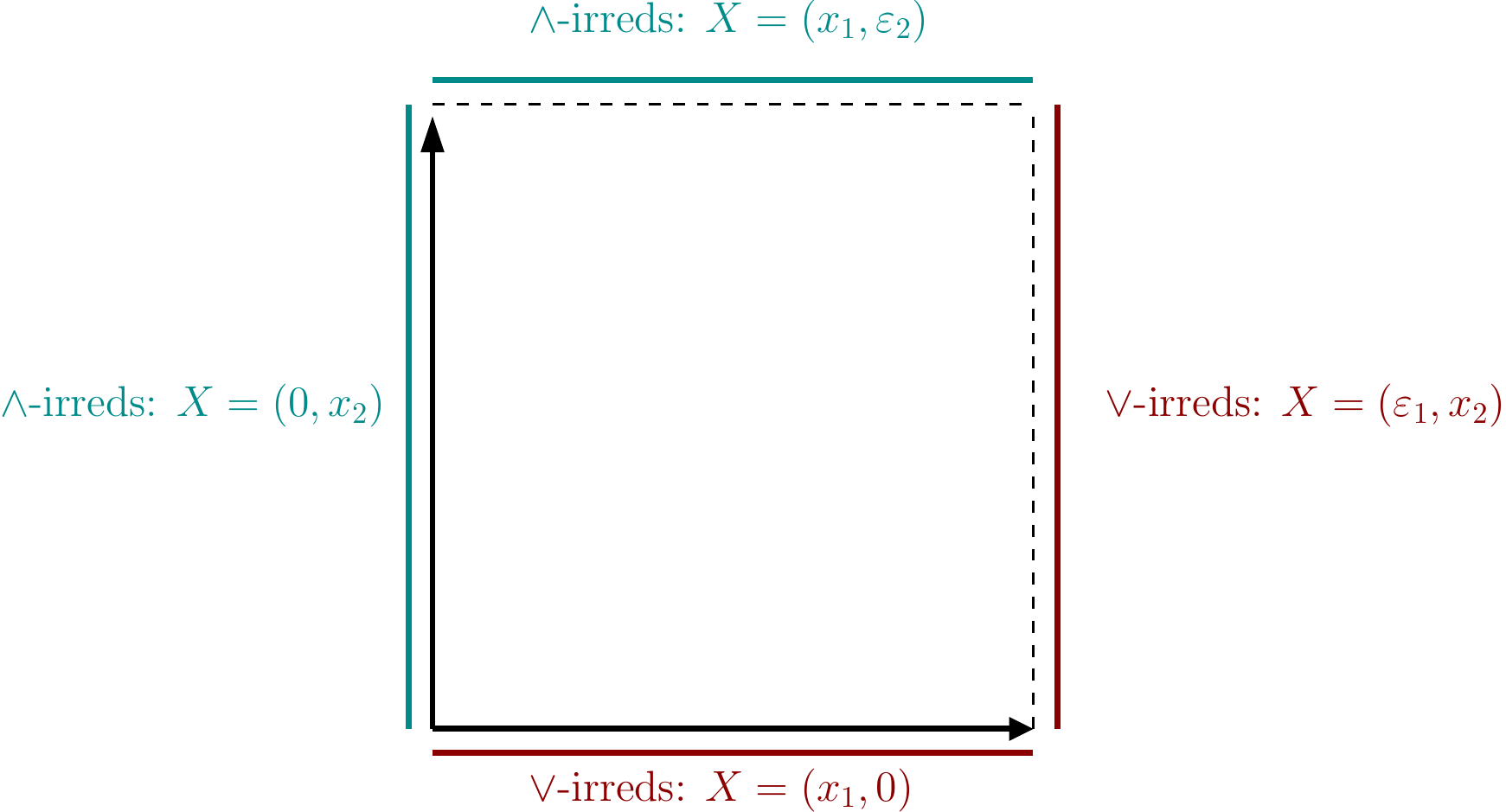}
\caption{$\wedge$-irreducible elements and $\vee$-irreducible elements in the algebra of all persistence diagrams.}
\label{figirred}
\end{figure}

\begin{proposition}
All prime principal ideals are of the form $\downarrow (x_1,\varepsilon_2) $ or $\downarrow (0, x_2)$, for some $(x_1,x_2)\in \HH$.
Dually, all prime principal filters are of the form $\uparrow (x_1,0)$ or $\uparrow (\varepsilon_1,x_2)$, for some $(x_1,x_2)\in \HH$.
\end{proposition}

\begin{proof}
Let $x\neq \perp$. 
It follows from the distributivity of $\HH $ that $x$ is $\vee-$irreducible if and only if $\HH \backslash \uparrow x$ is a prime ideal (cf.\cite{Gr71} pp.63).
Dually, $x$ is $\wedge $-irreducible element of $\HH$ if and only if $\HH \backslash \uparrow x$ is a prime filter.
On the other hand, $\HH \backslash \uparrow x$ is a prime filter if and only if $\downarrow x$ is a prime ideal of $\HH$ (cf.\cite{Gr71} pp.25).
Then, $x$ is $\wedge $-irreducible element of $\HH$ if and only if $\downarrow x$ is a prime ideal of $\HH$.
Hence, all prime principal ideals are determined by an element of the form $(x_1,\varepsilon_2) $ or $(0, x_2)$, for some $x=(x_1,x_2)\in \HH$.
\end{proof}

\begin{figure}
\includegraphics[page=1,width=0.5\textwidth]{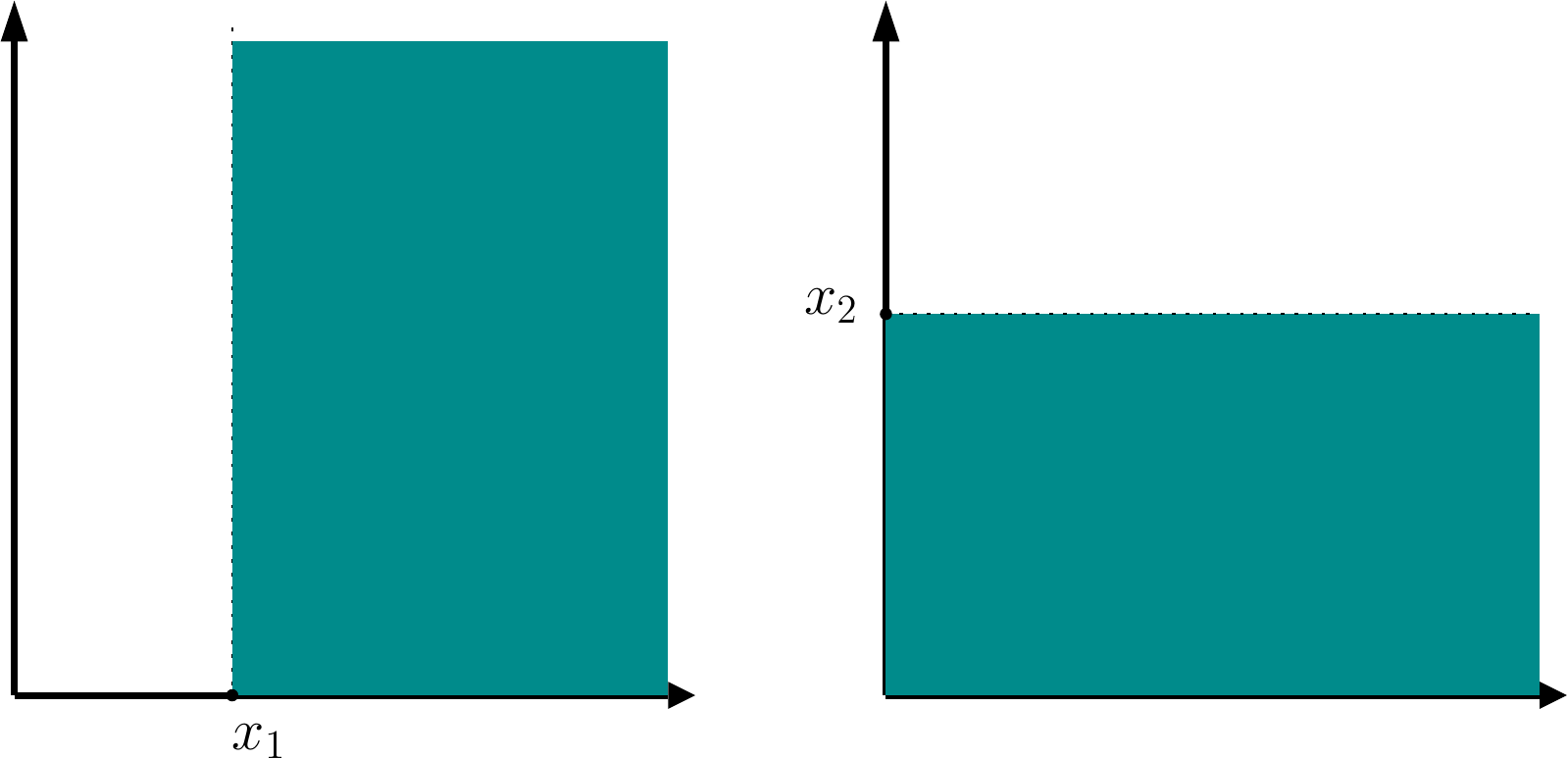}
\caption{The prime principal ideals of the algebra of lifetimes $\HH$, where it is clear the existence of two distinct behaviours.}
\label{compfilt}
\end{figure}   

The \textbf{category of locales} is determined by complete Heyting algebras (as objects) and morphisms between them preserving finite $\wedge$ and arbitrary $\vee$. 
A \textbf{sober space} is a topological space $X$ such that every irreducible closed subset of $X$ has a unique point $P$ whose closure is all of $X$. 
A locale $\LL$ is \textbf{spatial} (or \textbf{topological}) if each element of $\LL$ can be expressed as a meet of prime elements.

\begin{proposition}\label{spaciality}
The algebra of lifetimes, $\HH$, is a spatial locale.
\end{proposition}

\begin{proof}
Any element $P$ of the algebra of lifetimes $\HH$ is determined by the prime ideals corresponding to its coordinates: considering $P_x=(0,x)$ and $P_y=(y,0)$, $P=P_x\wedge P_y$ where the prime principal ideals correspondent to $P_x$ and $P_y$ are $\downarrow (0,x)$ and $\downarrow (y,0)$, respectively, represented in Figure \ref{compfilt}.   
\end{proof}

A \emph{lattice duality} describes the categorical equivalence between the category of topological spaces that are sober with continuous functions, and the category of locales that are spatial with appropriate homomorphisms (cf. \cite{Joh86}).
The dual space correspondent to the locale $\HH$ is made out of all its prime elements, i.e., all the elements in $\HH$ that determine principal prime ideals: $(0,x_2)$ and $(x_1,\varepsilon_2)$, for all $(x_1,x_2)\in \HH$. 
A basis of opens for the topology of that dual space is given by the vertical and horizontal filters $\uparrow (0,x_2)$ and $\uparrow (x_1,\varepsilon_2)$.
Each filter, horizontal or vertical, must contain $\top$ to give a sober space. 
The open sets are exactly the unions of these filters with different behaviours and thus we get a direct sum topology.
This is exhibited in Figure \ref{fig_alex}.

\begin{figure}
\includegraphics[page=1,width=0.3\textwidth]{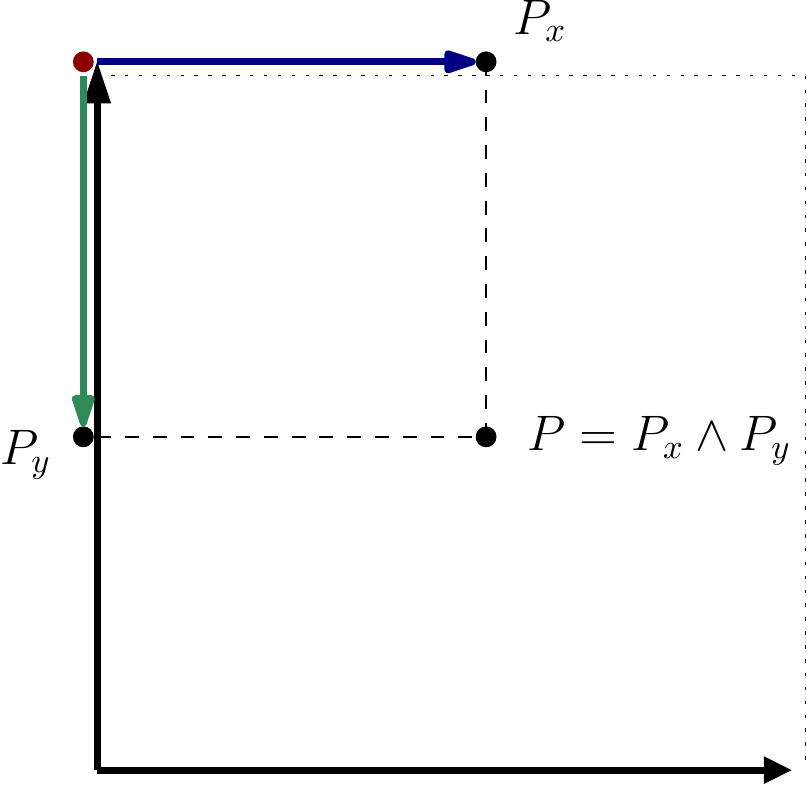}
\caption{An open of the dual space as the union of two filters $\uparrow P_x$ and $\uparrow P_y$ intersecting only in $\top$ and the point in the lattice it corresponds to.}
\label{fig_alex}
\end{figure}

The topological space described considers as points, birth and death time instances such as points $(t,t)$ in the diagonal, corresponding to an instance of time $t$ in some lifetime (open set) $(x,y)$ in $\HH$. 
This is in concordance with the fact that a real $t$ in $[0,\epsilon]$ is a point in the locale $([0,\epsilon];\wedge,\vee)$, and that $\HH=[0,\epsilon]^{op}Ê\times [0,\epsilon]$.
That provides us with the concept of point of the locale $\HH$ corresponding to a prime element in $\HH$ as earlier described.
The spaciality given by the proof of Proposition \ref{spaciality} is based on the recover of every element of the lattice $\HH$ as a lifetime corresponding to an interval $[t_1,t_2]$ with a birth time $t_1$ and a death time $t_2$.
Both of these instances are given by the corresponding open determined by both filters $\uparrow (t_1,\varepsilon_2)$ and $\uparrow (0,t_2)$. 

The discussed duality can provide more efficient techniques to solve problems living in the algebra of lifetimes $\HH$, by dealing with them in the dual space and transferring the solutions back to the lattice.
Such techniques and their implementation are subject of future work.  


\section{From local to global}
\label{sec:localglobal}

In the following section we will discuss the construction of sheaves over the algebra $\HH$, and will have a further look at the correspondent Grothendieck topos that will determine a framework of sets with a lifetime.
Recall that the barcode (or the persistence diagram) for a persistence module encodes the basis elements of the persistence module as pairs $(b,d)$ of a birth point and a death point of the basis element. 
This boils down to a persistence diagram or barcode being a multi-set of pairs of real numbers (or whatever the time set happens to be). 
Each element of the multi-set corresponds to a basis element of the module.
One of the ways this is being used is to say that for a given time point $x$, we can determine the local Betti number at that time by counting the number of points $(b,d)$ in the multi-set such that $b\leq x\leq d$. This can be visualised either as counting points in a quadrant or as counting bars intersecting a vertical line. There is nothing that keeps us from doing this for longer spans of query time intervals -- we can ask for points $(b,d)$ such that $b\leq x \leq y\leq d$ for some interval $(x,y)$. This produces Betti numbers that persist for \emph{at least} the time period $(x,y)$.
Now, the persistence Heyting algebra $\HH$ has as its elements intervals $(b,d)$, ordered by inclusion and with a Heyting algebra structure built according to the constructions in this text. Any actual barcode can be considered as a sheaf $\phi$ over this Heyting algebra $\HH$ such that $\phi(x,y)$ is the collection of basis elements in the persistence module that exist at least in the entire interval $[x,y]$.
In this setting, sheaves over $\HH$ encode (extended) barcodes or persistence diagrams. 
These sheaves can be considered as sets where each element is visible only over some time(-like) intervals.

%
%
The category of sheaves over $\HH$, denoted by $[\HH^{op},Set]$, is defined as follows:
\begin{itemize}
\item[(i)] Every object of this category is a sheaf $F:\HH^{op} \rightarrow Set$, and every morphism between two functors $F,G\in [\HH^{op},Set]$ is a natural transformation $\eta:F\rightarrow G$.
\item[(ii)] As these functors are contravariant, $\eta$ is an assignment to every object $x \in \HH$ of a morphism $\eta_x:F(x)\rightarrow G(x)$ in $Set$ (usually called the \textbf{component} of $\eta$ at $x$) such that for any poset arrow $f:x \to y$ in $\HH^{op}$, the following diagram commutes in $Set$:
\begin{center}
\begin{tikzpicture}[scale=.6]
\node (a) at (0,4) {$F(x)$};
\node (b) at (0,0) {$G(x)$};
\node (c) at (4,4) {$F(y)$};
\node (d) at (4,0) {$G(y)$};
\draw[arrows=-latex'] (a) -- (b) node[pos=.5,left] {$\eta_x$};
\draw[arrows=-latex'] (c) -- (d) node[pos=.5,right] {$\eta_y$};
\draw[arrows=-latex'] (c) -- (a) node[pos=.5,above] {$F(f)$};
\draw[arrows=-latex'] (d) -- (b) node[pos=.5,below] {$G(f)$};
\end{tikzpicture}
\end{center}
\end{itemize}

\begin{proposition}\label{consopen}
Consider the presheaf $\phi : \HH^{op} \rightarrow Set$ defined by the sections $\phi (x)=\downarrow x$, for all $x\in \HH$, and the restriction map $\chi_y^x:\phi(x)\rightarrow \phi (y) $ defined by $\chi_y^x (z)=z\wedge y$ for all $x,y,z\in \HH$ such that $y\leq x$. Then $\phi$ is a sheaf of sets over $\HH$.
\end{proposition}

\begin{proof}
It is clear that, in general, $z_i\wedge x_i\wedge x_j = z_i\wedge x_j$ so that the compatibility condition reduces to $z_i\wedge x_j=z_j\wedge x_i$.
Let us now show that $\phi$ is a separated presheaf.
Given $x=\bigvee_{i\in I} x_i$ in $\HH$ and $z,y\in \downarrow x$, the infinite distributivity law and the identity $z\wedge x_i=y\wedge x_i$ together imply that
\[
z=z\wedge x=z\wedge (\bigvee_{i\in I} x_i)=\bigvee_{i\in I} (z\wedge x_i)=\bigvee_{i\in I} (y\wedge x_i)=y.
\]   
We shall now see that $\phi$ constitutes a sheaf: given $x=\bigvee_{i\in I} x_i$ in $\HH$ and a compatible family $\set{z_i}_{i\in I}$ such that 
\[
z_i\wedge x_j=z_i\wedge x_i\wedge x_j=z_j\wedge x_i\wedge x_j=z_j\wedge x_i,
\]
put $z=\bigvee_{i\in I} z_i$ and notice that 
\[
z=\bigvee_{i\in I} z_i \leq \bigvee_{i\in I} x_i=x, \text{  i.e.,  } z\in \downarrow x.
\]
Fixing $i\in I$,
\[
z\wedge x_i  = (\bigvee_{j\in J} z_j)\wedge x_i = \bigvee_{j\in J} (z_j\wedge x_i) = \bigvee_{j\in J} z_i\wedge x_j= z_i\wedge (\bigvee_{j\in J} x_j)\nonumber = z_i \wedge x= z_i. \nonumber
\]
Hence, $z_{\mid x_i} = z_i$, as required.
\end{proof}

\begin{proposition}\label{toposcons1} 
The category of sheaves $[\HH^{op},Set]$ is a topos, with subobject classifier determined by the set of all sieves on $x$, that is,
\[
\Omega (x) =  \downarrow x \text{,  for all  } x\in \HH,
\]
with restriction maps $\downarrow x \to \downarrow y$ given by $z\mapsto z\wedge y$, for all $x,y\in \HH$ such that $y\leq x$. 
\end{proposition}

\begin{proof}
Let us notice that for $x\in \HH$, $\downarrow x$ corresponds to all the topological features that live while a feature with lifetime $x$ is alive.
Another presheaf that can be considered on $\HH$ is determined by $\phi (x)=\uparrow x$, for all $x\in \HH$, and by the restriction map $\chi_y^x:\phi(x)\rightarrow \phi (y) $ defined as $\chi_y^x (z)=z\vee x$ for all $y,x,z\in \HH$ such that $y\leq x$.
$\chi_y^x$ can be seen as the inclusion of $\phi(x)$ in the larger $\phi(y)$, corresponding to the example of the \emph{time sheaf of states of knowledge} described in \cite{Bar95}.
Recall that, in the poset category $\HH$, $\top$ is the terminal object. 
Now observe that the category $[\HH^{op},Set]$ also has terminal object: the constant functor which sends every element of $\HH$ to $\{ \top \}$ is a terminal object $\top$. 
It also has finite limits and exponentials due to the fact that the Yoneda embedding $y:\HH \rightarrow [\HH^{op},Set]$ preserves all products and exponentials in $\HH$, and the fact that limits and exponentials exist in any Heyting algebra (corresponding to the meet and arrow operations).
While the terminal object of $[\HH^{op},Set]$ is a constant sheaf, all exponential objects in $[\HH^{op},Set]$ constitute sheaves themselves.
And since limits commute with limits and every sheaf can be seen as an equaliser, the limits of sheaves are also sheaves (cf. \cite{Mac92}).  
Now, to show that $[\HH^{op},Set]$ is a topos we just need to show that $\Omega$ as defined in is the subobject classifier.
Observe that $\Omega (x)$ is the set of all (poset) arrows into a given $x\in \HH$. 
Consider the morphism between the terminal object $\top \in [\HH^{op},Set]$ and $\Omega\in [\HH^{op},Set]$, i.e., a natural transformation $t:\top \rightarrow \Omega$, that takes the maximal sieve for each $x\in \HH$, sending the point of $\top$ to the maximum element $x$ of $\downarrow x$. 
Now let $j:U\rightarrow E$ be a given monomorphism and take the morphism $\chi_j:E\rightarrow \Omega$ defined by $\chi_j(x)(e)$ to be the sieve of arrows into $x\in \HH$ that take $e\in E(x)$ back into the subobject $U$. 
Then $\chi_j$ ensures that the following commutative diagram is a pullback.
\begin{center}
\begin{tikzpicture}[scale=.3]
\node (a) at (0,4) {$U$};
\node (b) at (0,0) {$E$};
\node (c) at (4,4) {$\top$};
\node (d) at (4,0) {$\Omega$};
\draw[arrows=-latex', dashed] (a) -- (b) node[pos=.5,left] {$j$};
\draw[arrows=-latex'] (c) -- (d) node[pos=.5,right] {$t$};
\draw[arrows=-latex', dashed] (a) -- (c) node[pos=.5,above] {};
\draw[arrows=-latex'] (b) -- (d) node[pos=.5,below] {$\chi_j$};
\end{tikzpicture}
\end{center}
\end{proof}

The result above in Corollary \ref{toposcons1} provides us with enough structure to think of the appropriate model for persistence.
In that perspective, sets vary within a local section correspondent to an interval of time, while global section consider such sets and their variations correspondent to the totality of that lifetime.
In that sense, given an element $F$ of the topos $[\HH^{op},Set]$ and a lifetime $t=(t_1,t_2)\in \HH$, the set $F(t)$ corresponds to a family of sets alive during $[t_1,t_2]$. 
To the sheaves of sets over $\HH$ we call \textbf{$\HH$-sets}, and to the category of $\HH$-sets and natural transformations between them $[\HH^{op},Set]$ we shall call the \textbf{persistence topos}.


%
Skew distributive lattices are noncommutative generalisations of distributive lattices, studied in \cite{LKC13}. %
In detail, a \textbf{skew lattice} $(S;\wedge, \vee)$ is constituted by a set $S$ and associative binary operations $\wedge $ and $\vee $ satisfying the absorption identities $x=x\wedge (x\vee y)$ and $x=(y\wedge x)\vee x$ and their respective duals. 
The distributive versions of these algebras, named \textbf{strongly distributive} skew lattices, are the ones satisfying the identities $x\wedge (y\vee z)=(x\wedge y)\vee (x\wedge z)$ and $(x\vee y)\wedge z=(x\wedge z)\vee (y\wedge z)$. 
An example of such algebras are partial maps with skew lattice operations defined by
\[
f\wedge g=g_{\mid \dom f\cap \dom g} \text{  and  } f\vee g=f\cup g_{\mid \dom f - \dom g}.
\]
Furthermore, the quotient of this skew lattice by the congruence defined by the equality of domains (ie, $f\DD g$ if $\dom f = \dom G$) is a distributive lattice.

A partial order can be defined in a skew lattice by $x\leq y$ if $x=x\wedge y=y\wedge x$ (or, equivalently, $y=x\vee y=y\vee x$).
The order structure is relevant when considering the classical Priestley duality. A more general local Priestley duality has been established in \cite{CD83}, between distributive lattices with zero and ordered compact topological spaces $(\X, \tau, \leq)$ with a basis consisting of $\tau$-compact open downsets and such that, for all $x,y\in X$ such that $x\nleq y$ there exist disjoint clones $U\in \uparrow \tau$ and $V\in \downarrow \tau$ such that $x\in U$ and $y\in V$, called \textbf{local Priestley spaces}. A map in the context of that duality is any continuous order preserving map $g:X\rightarrow y$ such that $g^{-1}(Z)$ is $\tau _X$-compact, for all $Z$ $\tau _Y$-compact. 
Moreover, a duality between sheaves over local Priestley spaces and strongly distributive skew lattices was recently established in \cite{Bau13}, assigning an important role to strongly distributive skew lattices.
A refinement of the Priestley duality (and, consequently, also of local Priestley duality) was introduced in \cite{Pul88} showing that the complete Heyting algebra $\HH$ is dual to a particular ordered topological space $(\X,\tau,\leq)$ that is a totally order disconnected space for which $(X,\downarrow \tau)$ has a basis consisting of $\tau$-compact open downsets (i.e., a local Priestley space) satisfying that $U\in \downarrow \tau$ implies $cl(U)\in \downarrow \tau$. 
Consider a sheaf $\Gamma: \X^{op} \to Set$ given by $\OO_a \mapsto \Gamma_{\OO_a}$ and $\OO_a\subseteq \OO_b \mapsto \chi: \Gamma_{\OO_b}\mapsto \Gamma_{\OO_a}$.
Now, due to the duality between local Priestley spaces and strongly distributive skew lattices implies that $\Gamma$ is a distributive skew algebra of sections with the operations $\wedge$ and $\vee$ as defined above.
An investigation into these noncommutative algebras will yield additional insights into the persistence topos, specifically on the computations with the sections of the sheaves constituting that topos.



\section{Homology on variable sets}
\label{sec:homology}

A topos theoretic generalisation of the category of sets to the category of sets with lifetimes permits to compute homology on the underlying sets varying according to time intervals, providing tools for the unification of different flavors of persistence (we have highlighted standard, multidimensional and zigzag persistence). 
%
%
The idea of sets that change over time is one of the main building blocks of this research where we look at a topos as a category of Heyting algebra valued sets where the information of $\HH$ is encoded.
Homology will not be done directly on the topos $[\HH^{op},Set]:=Set^{\HH^{op}}$. 
For these matters we consider the composition of functors $\HH^{op}\to Top\to Vect$ where $Top$ is the category of topological spaces and homeomorphisms, and $Vect$ is the category of vector spaces and linear maps.
This is a refinement of the ideas from \cite{Bu13} that present a functor $Pos \to Top \to Vect$ where $Pos$ is a partially ordered set seen as a small thin category for which the objects are its elements and the unique morphism between two elements, when existing, is provided by the order structure in the sense that $x\mapsto y$ iff $x\leq y$. 
The site $\HH$ has enough structure to be the right model for the starting poset category, being general enough to comprehend all the index sets used in persistence. 

For computational purposes we do not use the general features of $Top$ but rather a more combinatorial flavoured category as the category of simplicial sets $Simp$ and the composition of functors $\HH^{op}\to Simp\to Vect$
In fact, we shall use a weaker (more computable) version of the category of simplicial sets - the category of semi-simplicial sets - but for now let us continue this discussion considering the simplicial category, denoted by $\Delta$, for which the objects are the sets $[n]$ and the morphisms are order preserving maps. Let us also consider the category of presheaves $\Delta^{op}\to Set$ called \emph{simplicial sets}. 
The next step is to consider simplicial objects over the site of lifetimes $\HH$, given by functors 
\[
\Delta^{op}\to [\HH^{op},Set] \text{  or equivalently by  } \HH^{op}\to [\Delta^{op},Set] 
\]
that explicitly translate our idea of having the simplicial sets $[\Delta^{op},Set]$ varying according to $\HH$.
Moreover, the equivalent expression of these functors as $\Delta^{op}\to [\HH^{op},Set]$ permits us the intuition that we are constructing the simplicial sets over the topos $[\HH^{op},Set]$ that is providing the appropriate generalisation of the category $Set$.
When considering simplicial homology we can distinguish a functor $Simp \to Vect$. 
Thus, when extending $Set$ to the category of $\HH$-sets we get $\HH$-valued simplicial sets, denoted by $\HH Simp$, and the functor $\HH Simp\to Vect$.
These are $\HH$-valued sheaves on the category $Vect$ of vector spaces and linear maps.
Respectively we call the sheaves over $\HH$ in $Vect$ \textbf{$\HH$-spaces}.

%

We conclude this section with an example of the computation of persistent homology in a simplicial scenario using our developed tools.
Consider the following $n$-dimensional faces and corresponding lifetimes
\[ x (0,5) ; y (0,4) ; z (1,3) ; d (1,4) ; e (1,3) ; f (1,3) ; t (2,3) \]
together with the following face maps:
\[
\begin{array}{rcl}
d_0:t\rightarrow d & d_0:d\rightarrow x & d_1:d\rightarrow y \\
d_1:t\rightarrow e & d_0:e\rightarrow y & d_1:e\rightarrow z \\
d_2:t\rightarrow f & d_0:f\rightarrow z & d_1:f\rightarrow x 
\end{array}
\] 
This example can be illustrated as the following time varying complex: 
\begin{center}
\includegraphics[scale=0.4, page=1,width=0.8\textwidth]{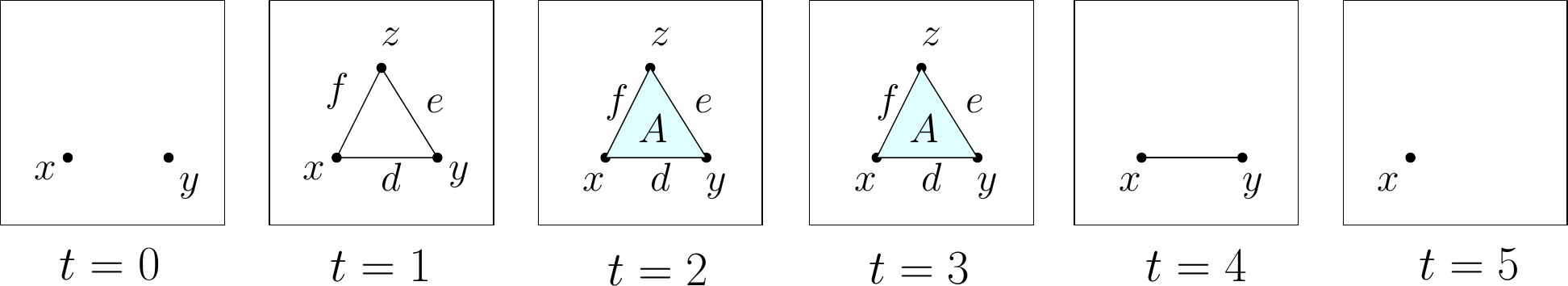}
\end{center}
Denote by $H_n(t)$ the homology group for dimension $n$ at time $t$.
For $t=0$, we get the complex $C_0=\langle x,y\rangle \rightarrow 0$ so that 
\[
H_0(0)=\langle x,y\rangle/0\cong \Z\oplus \Z
\]
while $H_1=0$.

For $t=1$ we consider the complex $C_1\rightarrow C_0\rightarrow 0$ so that 
\[
H_0(1)=\langle x,y,z\rangle/\langle x-y,y-z,z-x\rangle\cong \Z
\]
On the other hand, $\sigma(la+mb+nc)=l(y-x)+m(z-y)+n(x-z)$ so that this linear combination is zero if and only if $l=m=n$.
Thus, $Z_1=\langle a+b+c\rangle$ and therefore 
\[
H_1=\langle a+b+c\rangle/0\cong \Z
\]
For $t=2$ as well as for $t=3$ we consider the complex $C_2\rightarrow C_1\rightarrow C_0\rightarrow 0$ and we get that 
\[
H_0(2)=\langle x,y,z\rangle/\langle x-y,y-z,z-x\rangle\cong \Z \text{  and that  } H_1(2)=\langle a+b+c\rangle/\langle a+b+c\rangle\cong 0
\]
as $B_1=\sigma A=a+b+c$. Moreover, $H_2=0$.
For $t=4$ we have the complex $C_1\rightarrow C_0\rightarrow 0$ so that 
\[
H_0(4)=\langle x,y\rangle/\langle x-y\rangle\cong \Z
\]
while $H_1(4)\cong 0$ and $H_2(4)\cong 0$.
Finally, for $t=5$, $H_0(5)=\langle x\rangle/0\cong \Z$ and $H_2(5)\cong 0$.

Due to the sheaf structure we may glue the pointwise homology groups together consistently. 
Doing so, we obtain the following global information:
\begin{itemize}
\item for $t=[0,1[$ we have $H_0=\Z\oplus \Z$ changing to $H_0=\Z$ when $t=[1,5]$.
\item for $t=[1,2[$ we have $H_1=\Z$ changing to $H_1=0$ when $t=[0,1[\cup[2,5]$.
\end{itemize}

By computing homology over our time variable sets we are able to capture global information in a similar way to the barcodes in classical persistence.
In future work we will describe how to extract indecomposables from this information as well as explore algorithms for efficient concrete computation. 

%


\section*{Acknowledgements}
The production of this paper and correspondent research was positively influenced during the past year by the following researchers listed by alphabetical order:
Andrej Bauer,
Karin Cvetko-Vah,
Graham Ellis,
Maria Jo\~ao Gouveia,
Dejan Govc,
Ganna Kudryavtseva,
Primo\v z Moravec,
Jorge Picado.
To all of them we gratefully hold a word of appreciation. 
The authors would also like to acknowledge that this work was funded by the
EU Project TOPOSYS (FP7-ICT-318493-STREP).


\bibliographystyle{plain}


\end{document}